\documentclass[12pt]{amsart}
\usepackage{amsmath}
\usepackage{amssymb}
\usepackage{amsthm}
\usepackage{mathrsfs}
\usepackage{comment}
\usepackage{hyperref}

\usepackage[all,cmtip]{xy}\usepackage{xcolor}
\usepackage{enumerate}
\usepackage{bm}
\usepackage{dsfont}
\usepackage{mathtools}
\usepackage[cal=euler]{mathalfa}
\usepackage[top=1in, bottom=1.25in, left=1.25in, right=1.25in]{geometry}
\usepackage{parskip}

\hypersetup{colorlinks=true,linkcolor=magenta,citecolor=blue}

\DeclareMathAlphabet{\mathpzc}{OT1}{pzc}{m}{it}

\theoremstyle{plain}
\newtheorem{theorem}{Theorem}[section]
\newtheorem*{theorem*}{Theorem}

\newtheorem{lemma}[theorem]{Lemma}
\newtheorem*{claim*}{Claim}
\newtheorem{proposition}[theorem]{Proposition}

\newtheorem{corollary}[theorem]{Corollary}

\theoremstyle{definition}
\newtheorem{definition}[theorem]{Definition}
\newtheorem{defn}[theorem]{Definition}

\newtheorem{example}[theorem]{Example}

\numberwithin{equation}{section}
\numberwithin{figure}{section}

\newcommand{\val}{v}
\newcommand{\RR}{\mathcal{R}} 
\newcommand{\PP}{\mathcal{P}} 

\newcommand{\extK}{K} 
\newcommand{\normKF}{N_{K/F}}
\newcommand{\normKFim}{N_{K/F}(K^\times)}
\newcommand{\proot}{\xi}  

\newcommand{\Q}{\mathbb{Q}}
\newcommand{\N}{\mathbb{N}}
\newcommand{\F}{\mathbb{F}}

\newcommand{\Gal}{\mathrm{Gal}}

\newcommand{\Z}{\mathbb{Z}} 
\newcommand{\ep}{\varepsilon}

\newcommand{\ignore}[1]{}

\begin{document}
\title{A parametrization of nonassociative cyclic algebras of prime degree}

\author{Monica Nevins}
\address{Department of Mathematics and Statistics, University of Ottawa, Ottawa, Canada K1N 6N5}
\email{mnevins@uottawa.ca}
\thanks{The first author's research is supported NSERC Discovery Grant RGPIN-2020-05020.}

\author{Susanne Pumpl\"un}
\address{School of Mathematical Sciences, University of Nottingham,
Nottingham NG7 2RD}
\email{Susanne.Pumpluen@nottingham.ac.uk}



\date{\today}
\maketitle

\begin{abstract}
We determine and explicitly parametrize the isomorphism classes of nonassociative quaternion algebras over a field of characteristic different from two, as well as the isomorphism classes of nonassociative cyclic algebras of odd prime degree when the base field contains a primitive $m$th root of unity.
In the course of doing so, we prove that any two such algebras can be isomorphic only if the cyclic field extension and the chosen generator of the Galois group are the same.  As an application, we give a parametrization of nonassociative cyclic algebras of prime degree over a local nonarchimedean field $F$, which is entirely explicit under mild hypotheses on the residual characteristic.  In particular, this gives a rich understanding of the important class of nonassociative quaternion algebras up to isomorphism over nonarchimedean local fields.
\end{abstract}

\section{Introduction}

The general study of
the classification and construction of
nonassociative algebras is a classical subject, dating back to pioneering works of  \cite{A2, A3, D}.  It is a rich area, including such important examples as Jordan algebras which go back to \cite{Pas}, and semifields (\emph{i.e.}, nonassociative division algebras over finite fields).

In this paper, we give a parametrization of the isomorphism classes  of a particular class of nonassociative algebras first defined (slightly more generally) as Sandler semifields in \cite{San62}, and later studied over any base field (\emph{e.g.} in \cite{S12}):  nonassociative cyclic algebras of degree $m$ (Definition~\ref{D:cyclic}).  They are defined by a cyclic Galois field extension $K/F$ of degree $m$, a generator $\sigma$ of $\Gal(K/F)$, and an element $a\in K\smallsetminus F$, and are denoted $(K/F,\sigma,a)$.  In particular, when the degree $m$ is prime, and $F$ contains a primitive $m$th root of unity, a nonassociative cyclic algebra of degree $m$ over $F$ is necessarily a  division algebra \cite{S12}.
 These algebras are a direct generalization of associative cyclic algebras which, over local and global fields $F$, generate the Brauer group of (associative) central simple algebras over $F$. Associative central simple algebras similarly arise as the key objects in the theorems of Merkurjev and Suslin on the $m$-torsion of Brauer groups over arbitrary fields.

After deriving our explicit parametrizations (in the quadratic case in Theorem~\ref{T:p2} and in the case of extensions of prime degree $m$, assuming $F$ contains a primitive $m$th root of unity in Theorem~
\ref{Main:general}), we go on to apply our parametrization to nonassociative cyclic algebras over local nonarchimedean fields $F$.  These are the finite algebraic extensions of the $p$-adic numbers $\Q_p$, together with Laurent series over a finite field.
Local nonarchimedean fields lie at the intersection of real and finite fields:  they are complete locally compact fields that are suitable for use in analysis, but retain abundant number-theoretic properties related to their finite residue fields.  We exploit these features to give a classification that is elegant and completely explicit under mild hypotheses on the field: for nonassociative quaternion algebras, the relevant results are Theorems~\ref{T:explicit}, \ref{T:Q2} and \ref{T:p=2quad}; for higher prime degree, our general result follows from Proposition~\ref{le:explicit} and we present in detail the case of $m=3$ in Theorem~\ref{thm:mainII}.

In the course of deriving our results, we also prove the following counterintuitive result (Theorem~\ref{sigmadistinct}), valid for nonassociative cyclic algebras of arbitrary degree (and in direct contrast with the associative case).

\begin{theorem*}
Let  $F$ be an arbitrary field and let $m\geq 3$ be the degree of a  cyclic Galois extension $K/F$.  For any two distinct generators $\sigma_1 \neq \sigma_2$ of the Galois group $\Gal(K/F)$, and for any $a_1,a_2\in K\smallsetminus F$, we have
$$
(K/F,\sigma_1,a_1) \not\cong (K/F,\sigma_2,a_2).
$$
\end{theorem*}

Given the importance of associative cyclic algebras in the Brauer group of $F$, it is natural to expect that nonassociative cyclic algebras (which are pivotal examples of semiassociative algebras) will play a similarly central role in the recently defined semiassociative Brauer monoid \cite{BHMRV, P24}.
Nonassociative cyclic algebras  also recently appeared as first examples of residue $m$-hyperrings
  \cite[Example 4.10 (i), Remark 4.11]{R}. 

Furthermore, in the special case that  $m$ is equal to 2, the nonassociative cyclic algebras over $F$  are exactly all the four-dimensional unital algebras over $F$ which have a separable quadratic field extension of $F$ contained in their nucleus \cite{W}.
They are also identical to the nonassociative quaternion (division) algebras that are obtained by a generalized Cayley-Dickson doubling \cite{AS}.

In general, the classification of $n$-dimensional nonassociative algebras is a wild problem. Rough classifications up to the derivation types of the algebras, or up to their automorphism groups, are more tractable, and
  are possible for small dimensions, fixed base fields, or certain families of nonassociative algebras. Classifications up to isotopy, or even up to isomorphism,  are often difficult even for small dimensions and special base fields.
Investigations of
special classes of small dimensional algebras usually involve structure constants. For a comprehensive recent literature review, see \cite[Section 2.1]{Ka}.

Three- and four-dimensional nonassociative division algebras over a $p$-adic field were roughly classified up to isotopy in \cite{DG} and   Limburg \cite{L}
by studying the determinant of the matrix of their left multiplication and classifying it up to isometry. Unfortunately, this classification does not reveal the algebraic structure for some of the classes of algebras listed.

To our knowledge, ours are some of the first known parametrizations of a large class of nonassociative division algebras over nonarchimedean local fields $F$, and our work invites several interesting next directions.

For one, there should be a relationship between the nonassociative cyclic algebras arising from unramified extensions and those arising from the corresponding extensions of the finite residue field, through localization.   Note that nonassociative cyclic algebras are special cases of Petit algebras  \cite{P66, P68, LS}, and over finite base fields isotopic to Jha-Johnson semifields \cite{JJ}, which form one of the largest known class of semifields.
However, it is evident from examples that this localization depends on the parameters in ways that are not evident from the high-level theory.

We also believe that the explicit nature of our parameterization  will
help in tackling concrete open problems in the new theory of semiassociative algebras.
We note that nonassociative cyclic algebras
are also employed in the construction of space-time codes which are used in digital data transmission, and in linear code constructions; see, for example,  \cite{P15, PS15.4,  PS15}.
It would be very interesting to explore how to exploit the essentially binary nature of nonassociative cyclic algebras over $\F_2(\!(t)\!)$, for example, to the development of more efficient linear  codes.

The paper is organized as follows.
We set our notation and collect some main results on nonassociative cyclic algebras and nonarchimedean local fields in Sections
\ref{sec:prel} and \ref{sec:p-adic fields}.
In Section \ref{sec:quat} we prove our  parametrization (a classification up to isomorphism) of nonassociative quaternion algebras, first over general fields of characteristic different from $2$ (Theorem~\ref{T:p2}), and then specialized to  nonarchimedean local fields, including the characteristic $2$ case, which has a very different flavour (Theorem~\ref{T:p=2quad}). In particular, we have thus achieved full parametrizations of all exactly all four-dimensional unital algebras over $F$ which have a separable quadratic field extension of $F$ contained in their nucleus.

Section \ref{sec:prime} contains the two main results of this paper (Theorems \ref{sigmadistinct} and \ref{Main:general}). 
We then derive the parametrization over local nonarchimedean fields of residual characteristic different from $m$ (and containing $\mu_m$) employing Lemma \ref{le:explicit} and in particular for degree three in Section \ref{sec:degree3}.  We conclude in Section~\ref{sec:degree4} with a short exploration of the case of degree $4$ nonassociative cyclic algebras over a local nonarchimedean field such that $p\neq 2$.

\section{Notation and background} \label{sec:prel}

  We use $R^\times$ to denote the group of invertible elements of a unital associative ring $R$, and denote by $(R^\times)^n$ the subgroup $\{x^n\mid x\in R^\times\}$. Let $F$ be a field and
 let $\mu_m$ be the group of $m$th roots of unity in a fixed algebraic closure of $F$, when the characteristic of $F$ does not divide $m$.

\subsection{Nonassociative algebras} \label{subsec:nonassalgs}

An $F$-vector space $A$ is an
\emph{algebra} over $F$ if there exists an $F$-bilinear map $A\times
A\to A$, $(x,y) \mapsto x \cdot y$.  We denote this \emph{multiplication} in $A$ simply by the juxtaposition $xy$. An algebra $A$ is called
\emph{unital} if there is an element in $A$, denoted by 1, such that
$1x=x1=x$ for all $x\in A$. We will only consider unital finite-dimensional nonzero algebras.

Define the \emph{associator}  of $A$ by  $[x, y, z] =
(xy) z - x (yz)$.  The {\it left, middle and right  nucleus} of $A$ are defined as ${\rm
Nuc}_l(A) = \{ x \in A \, \vert \, [x, A, A]  = 0 \}$, ${\rm Nuc}_m(A) = \{ x \in A \, \vert \,
[A, x, A]  = 0 \}$ and  ${\rm Nuc}_r(A) = \{ x \in A \, \vert \, [A,A, x]  = 0 \}$, respectively. These are  associative
subalgebras of $A$ and their intersection
 ${\rm Nuc}(A) = \{ x \in A \, \vert \, [x, A, A] = [A, x, A] = [A,A, x] = 0 \}$ is  the {\it nucleus} of $A$.
 The
 {\it center} of $A$ is ${\rm C}(A)=\{x\in A\,|\, x\in \text{Nuc}(A) \text{ and }xy=yx \text{ for all }y\in A\}$.

An algebra $A\not=0$ is called a \emph{division algebra} if for any
$a\in A$, $a\not=0$, the left multiplication  with $a$, $L_a(x)=ax$,
and the right multiplication with $a$, $R_a(x)=xa$, are bijective.
 If $A$ has finite dimension over $F$, $A$ is a division algebra if
and only if $A$ has no zero divisors \cite[ pp. 15, 16]{Sch}.

\subsection{Nonassociative cyclic algebras} \label{subsec:nonass}

We now define our principal objects of study; we adapt the definition from \cite{S12}, where the opposite algebras were considered instead.

\begin{definition}\label{D:cyclic}
Let $K/F$ be a cyclic Galois field extension of degree $m$ and let $\sigma$ be a generator of $\Gal(K/F)$. Then for any $a\in K^\times$, the \emph{nonassociative cyclic algebra  $(K/F, \sigma, a)$ of degree $m$} is the unital algebra of dimension $m^2$ defined by
$$(K/F,\sigma,a)=\bigoplus_{s=0}^{m-1} Kt^s$$
 with  $F$-bilinear multiplication defined for $0\leq s,s' < m$ and $k,k'\in K$ by
\begin{equation}\label{E:defmultiplication}
(kt^s)(k't^{s'}) = \begin{cases}
k\sigma^s(k')t^{s+s'} & \text{if $s+s'<m$;}\\
k\sigma^s(k')at^{s+s'-m} & \text{if $s+s'\geq m$.}
\end{cases}
\end{equation}
 \end{definition}

 Note that for all $a\in K\smallsetminus F$ the algebra $(K/F,\sigma,a)$ is  not associative, not even power-associative, since for example  $t(t^m)=ta=\sigma(a)t$ but $(t^m)t = at$.  If we instead choose $a\in F^\times$, then the definition above yields the classical central simple associative cyclic algebra $(K/F,\sigma,a)$ of degree $m$ \cite[\S1.4]{J96}. In the following, we call a nonassociative cyclic algebra  $(K/F,\sigma,a)$  that is not associative (\emph{i.e.} $a\in K\smallsetminus F$) sometimes also a \emph{proper} nonassociative cyclic algebra.

 An equivalent way to define $(K/F,\sigma,a)$ involves twisted polynomials: Let $f(t)=t^m-d\in R=K[t;\sigma]$. Then $R_m=\{g\in K[t;\sigma]\,|\, {\rm deg}(g)<m\}$ together with the usual addition and the multiplication defined via
 $g\circ h=gh \,\,{\rm mod}_r f $, where the right hand side is simply the remainder of $gh$ after division by $f$ of the right,
 is the nonassociative cyclic algebra $(K/F,\sigma,a)$,  also denoted by $R/R(t^m-a)$
as it is a special case of a  Petit algebra \cite{P66}.
Note that when ${\rm deg}(g)+{\rm deg}(h)<m$, the multiplication in $(K/F,\sigma,a)$ is the same as the multiplication in $R$.
  If $a\in F^\times$ then $(K/F,\sigma,a)$ is  the classical central simple associative cyclic algebra $(K/F,\sigma,a)$ which can also be written as  the quotient algebra $K[t;\delta]/(f)$ \cite[(7),(9), (10)]{P66}, \cite[p.~19]{J96}. This justifies the notation $R/R(t^m-a)$ introduced by Petit.

We record some quick properties.

\begin{lemma}\cite{P66}
Let $A=(K/F, \sigma, a)$ with $a\in K\smallsetminus F$.  Then $C(A)=F$, ${\rm Nuc}_l(A)={\rm Nuc}_m(A)=K$ and ${\rm Nuc}_r(A)=\{g\in R/Rf\,|\, fg\in Rf\}$.
\end{lemma}

Note that $\{g\in R/Rf\,|\, fg\in Rf\}$ can be identified as the $0$-eigenspace of the right multiplication operator $R_f$, also called the \emph{eigenspace of $f$} \cite{P66}.

 In particular when $m$ is prime we must have ${\rm Nuc}_r((K/F,\sigma, a))=K$.
  We also know that $(K/F,\sigma, a)$ is a division algebra if and only if  ${\rm Nuc}_r((K/F,\sigma, a))$ is a division algebra, \emph{e.g.} see \cite[Proposition 4]{G}.

Nonassociative cyclic algebras of degree 2 are also called \emph{nonassociative quaternion algebras}, and were first described by Dickson \cite{D}.  They can also be constructed by a generalized Cayley-Dickson process ${\rm Cay}(K,a)$ \cite[Lemma 1]{AS}, as it coincides with the specialization to $m=2$ of Definition~\ref{D:cyclic}.

\subsection{Properties and isomorphisms} Let us now summarize some needed results from the literature.

\begin{theorem} \cite{P66} 
Suppose either that $m=2$ or $3$, or that $m\geq 5$ is a prime such that  $F$ contains a primitive $m$th root of unity. Let $a\in K\smallsetminus F$. Then $(\extK/F, \sigma,a)$ is a division algebra.
\end{theorem}

The case $m=2$ is distinguished:  nonassociative quaternion algebras are, up to isomorphism, the only four-dimensional unital division algebras over a field $F$ that have a quadratic field extension $K/F$ in their nucleus, and which are not associative \cite[Theorem 1]{W}.
We thus put a special emphasis on them in this paper.

When $m$ is not prime, we have the following statement.

\begin{theorem}\label{SteeleThm}  \cite[Theorem 4.4, Corollary 4.5]{S12}
Suppose that $(K/F,\sigma,a)$ is a nonassociative cyclic algebra where $K/F$ is of degree $m$. 
If $a$ does not lie in a proper subfield of $K/F$, then $(K/F,\sigma,a)$ is a division algebra.
\end{theorem}

\begin{proof}
If $1,a,\cdots, a^{m-1}$ are linearly independent over $F$, then $(K/F,\sigma,a)$ is a division algebra \cite[Theorem 4.4]{S12}.
 The elements $1,a,\cdots, a^{m-1}$ are linearly independent over $F$ if and only if  the minimal polynomial of $a$ over $F$ has degree at least $m$, whence $F(a)/F$ has degree at least $m$.  Thus this is condition is equivalent to having $K=F(a)$ and  thus that $a$ lies in no proper subfield of $K$.
\end{proof}

We now turn to the question of when two such algebras are isomorphic.

\begin{lemma}\label{Lem:differentextensions}
Suppose $K,K'$ are two distinct cyclic Galois extensions of a field $F$.  Then for any generators $\sigma,\sigma'$  of their respective Galois groups, and any $a\in K\smallsetminus F$, $a'\in K'\smallsetminus F$,
we have
$$(K/F,\sigma,a)\not\cong (K'/F,\sigma',a').$$
\end{lemma}

The proof is immediate, since  isomorphic algebras will have the same left nuclei. In the rest of the paper, we will rely on the following strong classification theorem \cite[Corollary 32]{BrownPumpluen2018}.

\begin{theorem}\label{T:classifyisomorphism}
Let $\extK$ be a cyclic Galois extension of $F$ and let $\sigma$ be a generator of $\Gal(K/F)$.  For $a,b\in \extK\smallsetminus F$ we have that $(K/F,\sigma,a)\cong (K/F,\sigma,b)$  if and only if
there exists some $\tau\in \Gal(\extK/F)$  such that
\begin{equation}\label{equivrelation}
a \in \tau(b)\normKFim.
\end{equation}
\end{theorem}

For each such extension $K$, we denote by $\sim$ the equivalence relation on $K\smallsetminus F$ induced by \eqref{equivrelation}, that is, $a\sim b$ if and only if there is some power $i\in\Z$ such that $a\in \sigma^i(b)\normKFim$.

\section{Background on p-adic fields} \label{sec:p-adic fields}

A local nonarchimedean field is a locally compact, nondiscrete field equipped with an ultrametric norm $\vert \cdot \vert$.  For each prime number $p$, there are precisely two kinds:  the finite algebraic extensions of the $p$-adic numbers $\Q_p$, which are the completion of a number field with respect to the $p$-adic norm; and the fields $\F_q(\!(t)\!)$ of Laurent series over a finite field of order $q=p^n$ for some $n\geq 1$, with $\vert t \vert<1$.  The former are called $p$-adic fields, and have characteristic zero.  An excellent resource for the summary given here is \cite[Chapter II]{Lang1994}; for examples, see \cite{LMFDB}.

Let $F$ be a local nonarchimedean field.  Its integer ring is its maximal compact open subring $\RR=\{x\in F : |x|\leq 1\}$.  This ring is local, with unique maximal ideal $\PP = \{x \in F : |x|<1\}$.  Its unit subgroup of invertible elements is then $\RR^\times = \RR \smallsetminus \PP$.  The quotient $\kappa :=\RR/\PP$, called the \emph{residue field } of $F$, is a finite field of order $q=p^n$ for some $n\in \N$ and prime $p$ (called the \emph{residual characteristic} of $F$).

\begin{example} If $F=\F_q(\!(t)\!)$ then $\RR = \F_q[\![t]\!]$, $\PP=( t )$, and $\kappa = \F_q$. Similarly, if $F=\Q_p$ then we have
$$
\Q_p^\times=\{ a=\sum_{i=N}^\infty a_i p^i : N\in \N, a_i \in \{0,1,\ldots,p-1\}\}
$$
with addition and multiplication computed ``with carrying" as one would for finite sums. Here the least $i$ such that $a_i\neq 0$ is the $p$-adic valuation of $a$, and in this case $\vert a\vert = p^{-i}$.  The $p$-adic integers $\RR=\Z_p$ is the set of all elements of valuation at least $0$.
\end{example}

More generally, we define the valuation of $a\in F^\times$ by $\val(a)=\min\{n\in \Z: a\in \PP^n\}$ and scale the norm so that $|a| = q^{-\val(a)}$.  We fix a generator $\varpi$ of $\PP$ for once and for all; it satisfies $\val(\varpi)=1$.

The fundamental relationship between $F$ and $\kappa$ is captured by Hensel's Lemma.  The following general version is adapted from \cite[Thm 7.3]{Eisenbud1995}. 

\begin{lemma}[Hensel's Lemma] \label{le:Hensel}
Let $f(x)\in \RR[x]$ be a polynomial and $f'(x)$ its formal derivative.  If $a\in \RR$ satisfies
$$
f(a) \equiv 0 \mod f'(a)^2\PP,
$$
then there exists a unique $b\in \RR$ satisfying
$$
f(b)=0 \quad \text{and}\quad b\equiv a \mod f'(a)\PP.
$$
\end{lemma}

\begin{example}\label{eg:Teichmuller}
The elements of $\kappa^\times$ are the roots of $f(x)=x^{q-1}-1$.  For each $\overline{a}\in \kappa^\times$, choose $a\in\RR$ to be some preimage.  Since $f(a)\equiv 0 \mod \PP$ and $f'(a)=(q-1)a^{q-2}\in \RR^\times$, Hensel's Lemma implies that there exists a unique $b\in \RR$ satisfying $b^{q-1}=1$ and such that $b\in a+\PP$.  The set of all such $b$ (sometimes called Teichm\"uller lifts of elements of $\kappa^\times$) are distinct and form the group of $(q-1)$th roots of unity $\mu_{q-1}$ in $F$.  Note that we similarly recover $\mu_n\cap F$ for any $n$ not divisible by $p$, and this group will have order $n$ exactly when $n|(q-1)$.
\end{example}

It is sometimes convenient to decompose $F^\times$ as the direct product of groups
\begin{equation}\label{E:productform}
F^\times \cong \mu_{q-1} \times (1+\PP) \times \Z
\end{equation}
via the map that associates to a triple $(b, u, n)$  the element $bu \varpi^n$.

Finite algebraic extensions $F$ are again local nonarchimedean fields and they come in two forms.  There is a unique (Galois) \emph{unramified extension} of each degree $n$, which is the splitting field of $x^{q^n}-x$.  We often denote it $L_n$.  Its residue field is $\kappa_{L_n}\cong \F_{q^n}$ and its maximal ideal $\PP_{L_n}$ is generated over $\RR_{L_n}$ by $\varpi$.

At the other extreme are the totally ramified extensions $K$ of $F$ of degree $n$, of which there are usually several (in fact, infinitely many if the characteristic of $F$ is $p=n$), and not all of which are Galois.  They each have the property that their residue field satisfies $\kappa_K \cong \kappa$ and $\varpi$ generates $\PP_K^n$.

A general finite algebraic extension $K$ of $F$ can be uniquely factored as an unramified extension $L/F$ of degree $f$ followed by a totally ramified extension $K/L$ of degree $e$, the \emph{ramification index}.  We say $K/F$ is \emph{tamely ramified} if $(e,p)=1$ and \emph{wildly ramified} if not.  The totally and tamely ramified extensions of degree $e$ are all obtained by adjoining to $F$ a root of $x^n-u\varpi$, for some choice of $u\in \RR^\times/(\RR^\times)^e$; these are Galois only if $F$ contains a primitive $e$th root of unity.

By local class field theory \cite[XI \S4]{Lang1994}, the Galois extensions of degree $n$ of $F$ are in one-to-one correspondence with the subgroups of index $n$ of $F^\times$, via the map that assigns $K/F$ to the subgroup $\normKFim$, which is the image of the norm map.  In particular, this implies that $F^\times/\normKFim$ is a finite group of order $n$ equal to the degree of the extension.

For example, if $K=L_n$ is an unramified extension, then upon restriction to the integer ring, the norm map
$\normKF : \RR_{K}^\times \to \RR^\times$
is surjective, and the group $F^\times/\normKFim$ is represented by the elements
$$
\{ 1, \varpi, \varpi^2, \cdots, \varpi^{n-1} \}.
$$
On the other hand, for a totally and tamely ramified extension, $F^\times/\normKFim \cong \RR^\times/(\RR^\times)^n$  and Hensel's Lemma identifies this group with $
\kappa^\times /(\kappa^\times)^n$.

These  features make classifying nonassociative cyclic algebras explicitly quite tractable.

\section{Nonassociative quaternion algebras}\label{sec:quat}

 In this section, we first derive some general results about nonassociative quaternion algebras and then apply these to the case of local nonarchimedean fields.  As noted, these algebras exhaust the four-dimensional unital nonassociative division algebras that contain a quadratic field extension in their nucleus.

\subsection{Isomorphism classes of nonassociative quaternion algebras in odd characteristic}
Throughout this section, let $F$ denote an arbitrary field of characteristic different from two.  We begin with an elementary lemma.

\begin{lemma}
Suppose the characteristic of $F$ is odd, or zero.  The distinct quadratic field extensions of $F$ are given by $F(\sqrt{c})$ as $c$ runs over the distinct nontrivial classes in $F^\times/(F^\times)^2$.
\end{lemma}

\begin{proof}
Let $\extK$ be a quadratic extension field of $F$, defined as the splitting field of an irreducible quadratic polynomial $m(x) = x^2+ax+b\in F[x]$.  Since $2\in F^\times$, by the quadratic formula, the roots of $m(x)$ are
$$
\alpha_{\pm} = -\frac{a}{2}\pm \frac12\sqrt{a^2-4b}.
$$
Since $a/2, 1/2\in F$, it follows that $F(\alpha_{\pm}) = F(\sqrt{a^2-4b})$ and that this is a field extension if and only if $c=a^2-4b \notin (F^\times)^2\cup \{0\}$.   The rest of the lemma follows.
\end{proof}

Applying Lemma~\ref{Lem:differentextensions}, we deduce that  if $A$ is a proper nonassociative quaternion algebra over $F$, then there exists a \emph{unique} (nontrivial) coset $c(F^\times)^2\in F^\times/(F^\times)^2$ and an element $a\in F(\sqrt{c})\smallsetminus F$ such that
$$
A \cong (F(\sqrt{c})/F, \sigma, a),
$$
where $\sigma\in \Gal(F(\sqrt{c})/F)$ denotes the unique nontrivial element of the Galois group of this quadratic extension.

We now refine this to an explicit classification of these nonassociative cyclic algebras.

\begin{theorem}\label{T:p2}
Suppose $\extK$ is a separable quadratic extension of a field $F$ of the form $\extK=F(\sqrt{c})$, for some $c\in F^\times\smallsetminus (F^\times)^2$ and let $\sigma\in \Gal(\extK/F)$ be nontrivial.  Let $\mathcal{C}_K$ denote a set of coset representatives of  $F^\times/\normKFim$.   Then the distinct isomorphism classes of the  nonassociative quaternion algebras with left nucleus $\extK$ are represented by $(\extK/F,\sigma,a)$ where $a$ is chosen from the set $\mathcal{S}(K)$ given by
$$
\mathcal{S}(K)=\begin{cases}
\{r\sqrt{c},  r+s\sqrt{c} \mid r\in \mathcal{C}_K, s\in F^\times/\{\pm1\}\} & \text{if $-1\in \normKFim$}\\
\{r'\sqrt{c}\mid r'\in \mathcal{C}_K/\{\pm1\}\} &\\
 \quad \cup \{ r+s\sqrt{c} \mid r\in \mathcal{C}_K, s\in F^\times/\{\pm 1\}\} & \text{if $-1\notin \normKFim$.}
\end{cases}
$$
\end{theorem}

\begin{proof}
We begin by noting that if $-1\notin \normKFim$, then the quotient group $F^\times/\normKFim$ contains a class represented by $-1$, whence the normal subgroup we denote $\{\pm 1\}$; this defines the expression $\mathcal{C}_K/\{\pm 1\}$.

Since $S\subset \extK\smallsetminus F$, each $a\in S$ defines a nonassociative cyclic algebra $(\extK/F,\sigma,a)$ with left nucleus $\extK$.
By Theorem~\ref{T:classifyisomorphism}, two nonassociative cyclic algebras $(K/F,\sigma,a)$ and $(K/F,\sigma,b)$ are isomorphic if  and only if $a\sim b$, that is, if
$$
a \in b\normKFim \cup \sigma(b)\normKFim.
$$
What we will show is that for every $b\in \extK\smallsetminus F$, there exists a unique choice of $a\in \mathcal{S}(K)$ such that either $b\in a\normKFim$ or $\sigma(b)\in a\normKFim$.    We begin with existence. 

Let $b\in \extK\smallsetminus F$.  Then there exist $b_0\in F, b_1\in F^\times$ such that $b=b_0+b_1\sqrt{c}$.

If $b_0=0$, then since 
$b_1\in F^\times$, we may write $b_1=rn$ for a unique choice of $r\in \mathcal{C}_K$ and $n\in \normKFim$, whence  $b\in r\sqrt{c}\normKFim$.

If $b_0\in F^\times$, then there exists a unique $r\in \mathcal{C}_K$ and $n\in \normKFim$ such that $b_0 = rn$.  Setting $s=b_1n^{-1}$, we have
$$
b = (r+s\sqrt{c})n \in (r+s\sqrt{c})\normKFim.
$$

Since $\sigma(r+s\sqrt{c}) = r-s\sqrt{c}$, the elements $r\pm s \sqrt{c}$ always represent the same isomorphism class of an algebra.  On the other hand, since  $\sigma(r\sqrt{c}) = -r\sqrt{c}$,  $\sigma$ preserves the class $r\sqrt{c}\normKFim$ if and only if $-1\in \normKFim$, in which case each of these elements represent pairwise nonisomorphic algebras.  If $-1\notin \normKFim$, then instead only the elements $r'\sqrt{c}$, for $r'\in \mathcal{C}_K/\{\pm 1\}$, are pairwise nonisomorphic.
  \end{proof}

For use in practice, we briefly offer an alternative parametrization of these nonassociative quaternion algebras.

\begin{corollary} \label{cor:mainI}
In the setting of Theorem~\ref{T:p2}, we may alternatively parametrize the distinct isomorphism classes of nonassociative quaternion algebras with left nucleus $K$ by elements of the set $\mathcal{S}'(K)$, where
$$
\mathcal{S}'(K) = \begin{cases}
\{t+r\sqrt{c} \mid r\in \mathcal{C}_K, t \in \{0\}\cup F^\times/\{\pm 1\}\} &\text{if $-1 \in \normKFim$};\\
\{t+r\sqrt{c}, \; r'\sqrt{c} \mid r\in \mathcal{C}_K, r'\in \mathcal{C}_K/\{\pm 1\}, t \in  F^\times/\{\pm 1\}\} &\text{if $-1 \notin \normKFim$}.
\end{cases}
$$
\end{corollary}

\begin{proof}
The elements of the form $r\sqrt{c}$ in $\mathcal{S}(K)$ and $\mathcal{S}'(K)$ are in bijection.  Suppose $r+s\sqrt{c}\in \mathcal{S}(K)$.  Then $s = r'n$ for a unique $r'\in \mathcal{C}_K$ and $n\in \normKFim$.  Therefore with $t=rn^{-1}$ we have
$$
r+s\sqrt{c} = (rn^{-1} + r'\sqrt{c})n \in (t+r'\sqrt{c})\normKFim.
$$
Since $s$ was only defined as an element of $F^\times/\{\pm 1\}$, the same is true of $t$, whence the bijection.
\end{proof}

\subsection{Case of local nonarchimedean fields of odd residual characteristic}
In this section, let $F$ be a nonarchimedean local field of residual characteristic $p\neq 2$; that is, we exclude the finite algebraic extensions of $\Q_2$ and of $\F_2(\!(t)\!)$.  We provide a fully explicit parametrization of the nonassociative quaternion algebras in this case, based on Theorem~\ref{T:p2}.

First note that the square classes in $F^\times$ are represented by
$$
F^\times/(F^\times)^2 = \{1, \ep, \varpi,\ep\varpi\}
$$
where $\ep\in \RR^\times\smallsetminus (\RR^\times)^2$ is a fixed nonsquare element of valuation zero, which may be thought of as a lift to $\RR^\times$ of a nonsquare element of the residue field via Hensel's Lemma \ref{le:Hensel}.
Therefore there are exactly three distinct quadratic extensions of $F$ in this case:
\begin{itemize}
\item $L_2 = F(\sqrt{\ep})$, the unique unramified quadratic extension;
\item $K_\varpi = F(\sqrt{\varpi})$, a ramified extension;
\item $K_{\ep\varpi} = F(\sqrt{\ep\varpi})$, a second ramified extension.
\end{itemize}

Secondly, note that local class field theory gives
$F^\times/\normKFim \cong \Gal(K/F)$.  Explicitly, when $K/F$ is quadratic, this means we may take our set of representatives to be
$$
\mathcal{C}_K= \begin{cases}
\{1,\varpi\} & \text{if $K=L_2$ is unramified over $F$};\\
\{1,\ep\} & \text{if $K/F$ is ramified.}
\end{cases}
$$

Thirdly: when $K=L_2$ is unramified, $-1 \in \RR^\times\subset \normKFim$, independent of $p>2$.  When $K/F$ is ramified, however, then $-1 \in \normKFim$ if and only if $-1\in (F^\times)^2$. Since $p\neq 2$, Hensel's Lemma \ref{le:Hensel} implies this occurs if and only if $-1$ is a square in the residue field $\F_q$.  Since $\F_q^\times$ is a cyclic group of order $q-1$, we infer directly that $-1$ is a square if and only if $q\equiv 1\mod 4$.

Finally, note that when $p\neq 2$, we choose a set of representatives for $F^\times/\{\pm 1\}$ via a choice of representatives of $\kappa^\times/\{\pm 1\}$.  Let $S_\kappa$ be such a set; then via the isomorphism $\kappa^\times \cong \mu_{q-1}\subset \RR^\times$ of Example~\ref{eg:Teichmuller} we can lift this to a subset of $\RR^\times$ contained in the subgroup $\mu_{q-1}$.  Then using \eqref{E:productform} we may take
$$
F^\times/\{\pm 1\} =\{ su_1\varpi^n \mid s\in S_\kappa, u_1\in 1+\PP, n\in \Z\} \cong (\mu_{p-1}/\mu_2) \times (1+\PP) \times \Z.
$$

Putting these together gives the following satisfyingly explicit and simple statement.

\begin{theorem}\label{T:explicit}
Let $F$ be a local nonarchimedean field of odd residual characteristic, and suppose its residue field $\kappa$ has $q$ elements.  Let $\extK$ be quadratic extension of $F$.  Then the distinct isomorphism classes of nonassociative quaternion algebras with left nucleus $\extK$ are given by $(\extK/F,\sigma,a)$ where $\sigma\in \Gal(\extK/F)$ is the nontrivial element and $a\in \mathcal{S}(K)$, where $\mathcal{S}(K)$ is given as follows:
\begin{enumerate}
\item If $\extK=F(\sqrt{\ep})$ is unramified, then
$$
\mathcal{S}(K) = \{\sqrt{\ep}, \;\varpi\sqrt{\ep},\; 1+s\sqrt{\ep},\; \varpi+s\sqrt{\ep} \mid s\in F^\times/\{\pm 1\}\}.
$$
\item If $\extK=F(\sqrt{\alpha})$ is ramified, where $\alpha \in \{\varpi,\ep\varpi\}$, then
\begin{enumerate}[(a)]
\item If If $q\equiv 1 \mod 4$ then
$$
\mathcal{S}(K) = \{\sqrt{\alpha}, \;\ep\sqrt{\alpha}, \;1+s\sqrt{\alpha},\; \ep+s\sqrt{\alpha} \mid s\in F^\times/\{\pm 1\}\}.
$$
\item If $q\equiv 3 \mod 4$ then
$$
\mathcal{S}(K) = \{\sqrt{\alpha}, \;\pm1+s\sqrt{\alpha} \mid s\in F^\times/\{\pm 1\}\}.
$$
\end{enumerate}
\end{enumerate}
\end{theorem}

\begin{proof}
We apply Theorem~\ref{T:p2}.
If $\extK/F$ is unramified, then $\mathcal{C}_K = \{1,\varpi\}$ and $-1 \in \normKFim$,  yielding the first case.  If it is ramified, then $\mathcal{C}_K = \{1,\varepsilon\}$.  By the preceding, if $q\equiv 1 \mod 4$ then $-1\in (F^\times)^2\subset \normKFim$, and this yields the second case.
In the last case, we have $-1\notin (F^\times)^2$, which implies we may without loss of generality choose $\varepsilon=-1$.  Thus  $\mathcal{C}_K=\{\pm 1\}$, giving the desired result.
\end{proof}

Note that for $p$-adic fields $F$, $p\not=2$, these algebras  lie in class one or two, or are new examples for class three or four in \cite{L}, since  proper nonassociative quaternion algebras  are not isotopic to twisted fields.

\subsection{Case of $2$-adic fields}

The classification of nonassociative quaternion algebras given in
Theorem~\ref{T:p2} applies also to the quadratic extensions of $2$-adic fields (meaning, finite algebraic extensions of $\Q_2$), since these have characteristic zero.  There are two differences.

 For one, when $F$ has residual characteristic $2$, we have
$$
|F^\times/(F^\times)^2| = 2^{e+2}
$$
were $e$ is the ramification degree of $F$ over $\Q_2$.  That is, the number of distinct quadratic extensions of $F$ increases exponentially with its absolute ramification index. As these give the distinct choices for the left nucleus of a nonassociative quaternion algebra, the number of isomorphism classes of nonassociative quaternion algebras  increases exponentially in the ramification index as well.  Nonetheless,  for any given choice of $F$, we can generate an explicit list  of representatives of $\RR^\times/(\RR^\times)^2$ using Hensel's Lemma~\ref{le:Hensel}, and then the distinct quadratic extensions are obtained as $F(\sqrt{c})$ as $c$ varies over the nontrivial elements of the product
$$
\{1,\varpi\}\times \RR^\times/(\RR^\times)^2.
$$

For another, while local class field theory implies that $\vert\mathcal{C}_K\vert = \vert F^\times/\normKFim\vert = 2$, it can be nontrivial to general an explicit representative, such as would be obtained from the Artin reciprocity map.  Of course, in degree two this is simply equivalent to identifying an element of $F^\times$ that is not in the image of the norm map.

To give the flavour of the parametrization of isomorphism classes of nonassociative quaternion algebras in this case, we present the example of $F=\Q_2$ in detail.
\begin{theorem}\label{T:Q2}
Let $F=\Q_2$.  Then $F$ has seven distinct quadratic extensions $K$, enumerated in Table~\ref{Table:Q2}. 
The isomorphism classes of nonassociative quaternion algebras with left nucleus $\extK$ are represented by $(\extK/F,\sigma,a)$ where $\sigma\in \Gal(\extK/F)$ is the nontrivial element and $a\in \mathcal{S}(K)$, where $\mathcal{S}(K)$ is given as follows:
\begin{enumerate}[(a)]
\item If $\extK=\Q_2(\sqrt{-3})$, then this extension is unramified and
$$
\mathcal{S}(K) = \{\sqrt{-3}, \;2\sqrt{-3},\; 1+s\sqrt{-3},\; 2+s\sqrt{-3} \mid s\in \Q_2^\times/\{\pm 1\}\}.
$$
\item If $\extK=\Q_2(\sqrt{\alpha})$ with $\alpha \in \{2,-6\}$, then this extension is ramified and $-1\in \normKFim$.  We have
$$
\mathcal{S}(K) = \{\sqrt{\alpha}, \;3\sqrt{\alpha}, \;1+s\sqrt{\alpha},\; 3+s\sqrt{\alpha} \mid s\in \Q_2^\times/\{\pm 1\}\}.
$$
\item If $\extK=\Q_2(\sqrt{\alpha})$ with $\alpha \in \{-1,-2,3,6\}$, then this extension is ramified and $-1\notin \normKFim$.  We have
$$
\mathcal{S}(K) = \{\sqrt{\alpha}, \;\pm1+s\sqrt{\alpha} \mid s\in F^\times/\{\pm 1\}\}.
$$
\end{enumerate}
\end{theorem}

\begin{proof}
We record in Table~\ref{Table:Q2} several facts about the quadratic extensions of $\Q_2$, as follows.  Using Hensel's Lemma, we deduce that there are no solutions to $x^2=\alpha$ for $\alpha\in \{-1,-3,3\}$ but that every element of $1+4\Z_2$ is a square; moreover, no element of valuation equal to $1$ may have a square root in $F$.  We deduce that $\Q_2^\times/(\Q_2^\times)^2 = \{\pm 1, \pm 2, \pm 3, \pm 6\}$.  Thus the first column lists the distinct quadratic extensions of $\Q_2$, with the unramified extension (the one containing the cube roots of unity) being listed first.  The second column is taken from \cite[p34]{Shimura2010} where the image of the norm map was computed explicitly on a case-by-case basis.  In the third column we record a convenient choice of element $\gamma$ such that $\mathcal{C}_K=F^\times/\normKFim = \{1,\gamma\}$.

The theorem is now a direct application of Theorem~\ref{T:p2}.  To specify the cases of that theorem that apply, we explicitly record in the fourth column of Table~\ref{Table:Q2} whether or not $-1$ is in the image of the norm map, and in the final column put a label to the corresponding case of the statement of Theorem~\ref{T:Q2}.
\begin{table}[ht]
\caption{The distinct quadratic extensions of $\Q_2$}
\label{Table:Q2}
\begin{tabular}{|c|c|c|c|c|}
\hline $K=F(\sqrt{c})$ & $\normKFim/(F^\times)^2$ & $\mathcal{C}_K$ 
& $-1\in \normKFim$? & Case\\
\hline
\hline
$\Q_2(\sqrt{-3})$ & $\{\pm 1, \pm 3\}$ & $\{1,2\}$ & yes & (a)\\
$\Q_2(\sqrt{-6})$ & $\{\pm 2, \pm 6\}$ & $\{1,3\}$ & yes& (b)\\
$\Q_2(\sqrt{2})$ & $\{\pm 1, \pm 2\}$ & $\{1,3\}$ & yes& (b)\\
$\Q_2(\sqrt{-1})$ & $\{1,-2,-3,6\}$ & $\{\pm 1\}$ & no & (c)\\
$\Q_2(\sqrt{-2})$ & $\{1,2,3,6\}$ & $\{\pm 1\}$ & no&(c)\\
$\Q_2(\sqrt{3})$ & $\{1,-2-,3,6\}$ & $\{\pm 1\}$ & no&(c)\\
$\Q_2(\sqrt{6})$ & $\{1,-2,3,-6\}$ & $\{\pm 1\}$ & no&(c)\\
\hline
\end{tabular}
\end{table}
\end{proof}

\subsection{Case of Laurent series over a field of characteristic two}

Suppose now that we are in the last remaining case, that of characteristic two.

Note first that extensions of the form $F(\sqrt{c})$ for any $c\notin (F^\times)^2$ are (purely) inseparable. Instead, the distinct Galois quadratic extensions over a field $F$ of characteristic $2$ are obtained, by the Artin--Schreier Theorem \cite[Ch VI Thm 6.4, 8.3]{LangAlgebra}, as
$$
K = F(x)/( x^2+x+c)
$$
 where $c\in F$ ranges over the nontrivial cosets of the additive subgroup $\mathfrak{P}_2=\{z^2+z\mid z\in F\}$. Note that if $\alpha\in K$ is a root of $x^2+x+c$ then so is $\alpha+1$, so $K=F(\alpha)$.

In this section, we let $F$ be a local nonarchimedean field of characteristic two.  Then we have $F=\F_{2^f}(\!(t)\!)$, for some $f\in \N$.

While $F$ admits a unique quadratic unramified Galois extension, it has infinitely many distinct ramified Galois quadratic extensions. For example, since each element of $\mathfrak{P}_2 \smallsetminus\RR$ must have \emph{even} valuation, it follows that the elements of $\{t^{-2k-1}: k\in \N\}$, whose distinct differences have odd valuation, represent infinitely many distinct cosets of $\mathfrak{P}_2$.

The classification of the distinct nonassociative quaternion algebras over $F$ correspondingly assumes a different flavour.  For one, there will be infinitely many distinct choices for the left nucleus, corresponding to these various distinct quadratic Galois extensions.  On the other hand, the parametrization of the isomorphism classes corresponding to a fixed quadratic extension admits a simpler description, as follows.

\begin{theorem}\label{T:p=2quad}
Suppose $\extK=F(\alpha)$ is a separable quadratic extension of a local field $F$ of characteristic $2$.  Let $\gamma$ be a nontrivial element of $F^\times/\normKFim$ and  $\sigma$ the  element of $\Gal(\extK/F)$ such that $\sigma(\alpha)=\alpha+1$.  The distinct isomorphism classes of nonassociative quaternion algebras with left nucleus $\extK$ are represented by $(\extK/F,\sigma,a)$ where $a$ is chosen from the set $\mathcal{S}(K)$ given by
$$
\mathcal{S}(K)=\{ s + \alpha, \gamma(s+\alpha) \mid s\in F/\{0,1\}\}
$$
where $F/\{0,1\}$ denotes the quotient of the additive group of $F$ by the two-element prime field $\F_2$.
\end{theorem}

\begin{proof}
Let $m(x) = x^2+x+c$ be the minimal polynomial of $\alpha$.  Its other root is $\alpha+1$, and these are linearly independent over $F$.  Thus
$$
\{\alpha, \alpha+1\}
$$
forms a basis for $\extK/F$.  Let us evaluate the equivalence relation $\sim$ arising from Theorem~\ref{T:classifyisomorphism} on $K\smallsetminus F$. 
Write $a=a_0+a_1\alpha$; then $a_1\neq 0$ since $a\notin F$.  Thus  there exists a unique $r\in \{1,\gamma\}$ and $n\in \normKFim$ such that $a_1=rn$.  Setting $s=a_0n^{-1}$, it follows that $a\in (s+r\alpha)\normKFim$.

Since $\sigma(s+\alpha) = (s+1) + \alpha$ and $\sigma(s+\gamma \alpha) = (s+\gamma) + \gamma \alpha$, and since the characteristic of $F$ is $2$, it follows that the distinct isomorphism classes are parametrized by the set
$$
\{ s_1 + \alpha, s_\gamma +\gamma \alpha \mid s_1\in F/\{0,1\}, s_\gamma \in F/\{0,\gamma\}\}.
$$
As the map sending $a\in F$ to $a\gamma \in F$ is an additive group automorphism sending $\{0,1\}$ to $\{0,\gamma\}$, we deduce the final form of the parameter set $\mathcal{S}(K)$.
\end{proof}

\section{Nonassociative cyclic algebras of degree $m>2$}\label{sec:prime}

In this section, we begin with some results for nonassociative cyclic algebras of degree $m>2$, before specializing to the case that $m$ is prime (and $F$ contains a primitive $m$th root of unity), concluding with an application to the case of local nonarchimedean fields.

\subsection{Different generators of $\Gal(K/F)$ give nonisomorphic algebras} Let $F$ be an arbitrary field such that $F$ admits a cyclic Galois extension $K$ of degree $m>2$.
By Lemma~\ref{Lem:differentextensions}, the choice of field $K$ is an invariant of the isomorphism class.

In the case of \emph{associative} cyclic algebras, that is, when $a\in F^\times$, one has $(K/F,\sigma,a)\cong (K/F,\sigma^k,a^k)$ \cite[15.1 Corollary a]{Pierce1982}, that is, different choices of generator of the Galois group yield isomorphic algebras.

We show that, in stark contrast, nonassociative cyclic algebras  corresponding to distinct choices of generator $\sigma\in \Gal(K/F)$ are \emph{never} isomorphic.

\begin{theorem}\label{sigmadistinct}
Let  $F$ be an arbitrary field and let $m\geq 3$ be the degree of a  cyclic Galois extension $K/F$.  For any two distinct generators $\sigma_1 \neq \sigma_2$ of the Galois group $\Gal(K/F)$, and for any $a_1,a_2\in K\smallsetminus F$, we have
$$
(K/F,\sigma_1,a_1) \not\cong (K/F,\sigma_2,a_2).
$$
\end{theorem}

\begin{proof}
Suppose to the contrary that there exists an isomorphism $\varphi\colon (K/F,\sigma_1,a_1) \to (K/F,\sigma_2,a_2)$; this is an isomorphism of $F$-vector spaces satisfying $\varphi(uv)=\varphi(u)\varphi(v)$ for all $u,v\in (K/F,\sigma_1,a_1)$. In particular, $\varphi$ restricts to a field automorphism of $K$; since the Galois group is cyclic, and thus abelian, this implies $\varphi$ commutes with $\sigma_i$ on $K$, for $i\in \{1,2\}$.

By  Definition~\ref{D:cyclic},
$(K/F,\sigma_i,a_i)=\bigoplus_{s=0}^{m-1} Kt_i^s$ with an $F$-bilinear multiplication defined for $0\leq s,s' < m$ and $k,k'\in K$ by
$$
(kt_i^s)(k't_i^{s'}) = \begin{cases}
k\sigma_i^s(k')t_i^{s+s'} & \text{if $s+s'<m$;}\\
k\sigma_i^s(k')a_it_i^{s+s'-m} & \text{if $s+s'\geq m$.}
\end{cases}
$$
In particular, we infer that the left $K$-submodule $Kt_i^s$ can be characterized as
$$
Kt_i^s = \{u\in (K/F,\sigma_i,a_i) \mid \forall k\in K, uk = \sigma_i^s(k) u\}.
$$
Since $\sigma_1,\sigma_2$ both generate $\Gal(K/F)$, there is some $1<j<m$, with $(j,m)=1$, such that $\sigma_1 = \sigma_2^j$.  It follows that $\varphi(t_1)=kt_2^j$ for some $k\in K$.

Let $\ell$ be the least positive integer such that $\ell j>m$; since $\sigma_1\neq \sigma_2$, we have $\ell<m$.  For any index $s\geq 1$, set $\gamma_s(k) = k\sigma_2^j(k)\sigma_2^{2j}(k)\cdots \sigma_2^{(s-1)j}(k)$.  Then one can show by induction that
$$
\varphi(t^s) =
\begin{cases}
\gamma_s(k)t_2^{sj} &\text{if $1\leq s<\ell$;}\\
\gamma_\ell(k)a_2t_2^{\ell j - m} &\text{if $s=\ell$.}
\end{cases}
$$
Since $1\leq \ell<m$, we have
$$
t_1 \cdot t_1^\ell = t_1^\ell \cdot t_1 = \begin{cases}
t_1^{\ell+1} & \text{if $\ell+1<m$};\\
a_1 & \text{if $\ell+1=m$.}
\end{cases}
$$
However, we have
$$
\varphi(t_1^\ell)\varphi(t_1) =  \gamma_{\ell}(k) a_2 t_2^{\ell j-m}\cdot kt_2^j  =
\begin{cases}
\gamma_{\ell+1}(k)a_2t_2^{\ell j + j-m} & \text{if $(\ell+1)j<2m$;}\\
\gamma_{\ell+1}(k)a_2^2t_2^{\ell j + j-2m} & \text{if $2m\leq (\ell+1)j$}
\end{cases}
$$
whereas
$$
\varphi(t_1)\varphi(t_1^\ell) = kt_2^j \cdot \gamma_{\ell}(k) a_2 t_2^{\ell j-m} =
\begin{cases}
\gamma_{\ell+1}(k)\sigma_2^j(a_2)t_2^{\ell j + j-m} & \text{if $(\ell+1)j<2m$;}\\
\gamma_{\ell+1}(k)\sigma_2^j(a_2)a_2t_2^{\ell j + j-2m} & \text{if $2m\leq (\ell+1)j$}.
\end{cases}
$$
It follows that $\varphi$ is well-defined only if $\sigma_1(a_2)=\sigma_2^j(a_2)=a_2$, meaning that $a_2$ lies in the fixed field of $\sigma_1$. But this is impossible: $\sigma_1$ generates the cyclic Galois group, and $a_2\notin F$, by construction. 
\end{proof}

\subsection{An explicit parametrization of isomorphism classes}

Now suppose that  $m$ is an odd prime and that
 $F$ contains a primitive $m$th root of unity. If $F$ has prime characteristic $p$ then this condition implies that $gcd(m,p)=1$. Kummer theory \cite[Ch VI Thm 6.2, 8.2]{LangAlgebra} ensures that the Galois extensions of $F$ of degree $m$ are in bijection with the subgroups of  $F^\times/(F^\times)^m$ of order $m$.

By Theorem~\ref{sigmadistinct}, we may fix a choice of cyclic field extension $K/F$ of degree $m$ and a choice of generator $\sigma$ of $\Gal(K/F)$.  We wish to classify the algebras $(K/F,\sigma,a)$, as $a$ runs over $K\smallsetminus F$, up to isomorphism.

Let $\zeta\in F$ be a primitive $m$th root of unity.
By Kummer theory, $K$ is the splitting field of some polynomial $x^m-b$, with $b\notin (F^\times)^m$.  Thus we may choose $\beta\in K$ to be a root of $x^m-b$ for which $\sigma(\beta) = \zeta \beta$.
Then $$
\{1, \beta, \beta^2, \ldots,\beta^{m-1}\}
$$
is an $F$-basis for $K$ with the property that $\sigma(\beta^k)=\zeta^k\beta^k$ for all $0\leq k <m$.

We first prove that the isomorphism classes of cyclic algebras $(K/F,\sigma,a)$ admit a coarse partition into $2^{m-1}$ subsets with respect to this choice of basis, inspired by the classification of quaternion algebras in Theorem~\ref{T:p2}.

\begin{defn}
Let $\mathcal{I} = \mathcal{P}(\{0, 1, \ldots, m-1\}) \smallsetminus \{\{0\},\emptyset\}$ be the set of nonempty subsets of $\{0,1, \ldots,m-1\}$, excluding the set $\{0\}$.  For each $I \in \mathcal{I}$, define
$$
K(I) = \{ \sum_{i=0}^{m-1} a_i \beta^i \mid \forall i\in I, a_i\in F^\times, \forall i\notin I, a_i=0\} \subset K\smallsetminus F.
$$
\end{defn}

Thus for example if $m=3$, $K(\{1,2\}) = \{ a_1\beta + a_2\beta^2 \mid a_1,a_2 \in F^\times\}$.   It follows directly that
$$
K \smallsetminus F = \bigcup_{I\in \mathcal{I}} K(I),
$$
and that these sets $K(I)$ are invariant both under the action of $\sigma$, and under multiplication by elements of $F^\times$. 
Applying Theorem~\ref{T:classifyisomorphism} gives the following lemma.

\begin{lemma}\label{IJdistinct}
Suppose $I, J\in \mathcal{I}$ are distinct.  Then for any $a\in K(I)$, $b\in K(J)$, we have
$$
(K/F,\sigma,a) \not\cong (K/F,\sigma,b).
$$
\end{lemma}

To further refine these partitions, we require the following definition.

\begin{defn}
Let $I\in \mathcal{I}$ be such that $\vert I \vert =k+1\geq 2$. Write $F^{[\times k]}$ for the $k$-fold direct product $F^\times\times F^\times \times \cdots \times F^\times$.  If the elements of $I$ are $i_0<i_1<\cdots <i_k$ then let
$$
\Delta_I = \{ (\zeta^{s(i_1-i_0)}, \zeta^{s(i_2-i_0)}, \ldots, \zeta^{s(i_k-i_0)}) \in F^{[\times k]} : 1\leq s\leq m\}
$$
where $\zeta$ is any primitive $m$th root of unity.  Write $F^{[\times k]}/\Delta_I$ for any fixed choice of representatives for these cosets. For any fixed choice of $a_{i_0}\in F^\times$, define
$$
K(I;a_{i_0}) = \{ a_{i_0}\beta^{i_0} + \sum_{j=1}^{k}a_{i_j}\beta^{i_j} : (a_{i_1},a_{i_2}, \ldots, a_{i_k})\in F^{[\times k]}/\Delta_I \}.
$$
Finally, set $K(I;a_{i_0})=\{a_{i_0}\beta^{i_0}\}$ when $I=\{i_0\}\in \mathcal{I}$.
\end{defn}

We may now state our main classification theorem for cyclic extensions of odd prime degree.

\begin{theorem} \label{Main:general}
Suppose $m$ is an odd prime and $F$ contains a  primitive $m$th root of unity $\zeta$.  Then the distinct isomorphism classes of nonassociative cyclic algebras of degree $m$ over $F$ are represented by
$$
(K/F,\sigma,a)
$$
where:
\begin{itemize}
\item $K$ is one of the distinct cyclic Galois field extensions of $F$ of degree $m$, and $\mathcal{C}_K$ denotes a set of representatives of $F^\times/\normKFim$;
\item $\sigma$ is a generator of $\Gal(K/F)$; and
\item $a \in \mathcal{S}(K) \subset K\smallsetminus F$
\end{itemize}
where $\mathcal{S}(K)$ is defined as follows.

If $\zeta\in \normKFim$, then we take
$$
  \mathcal{S}(K)=\bigcup_{I\in \mathcal{I},a_{i_0}\in \mathcal{C}_K} K(I;a_{i_0}).
  $$
whereas otherwise, we may take $\mathcal{C}_K = \mu_m$ and then
$$
\mathcal{S}(K) = \bigcup_{I\in \mathcal{I}:i_0=0, a_{0}\in \mu_m}  K(I;a_{0}) \;\cup\; \bigcup_{I\in \mathcal{I}:i_0>0} \{1+ \sum_{j=1}^{|I|-1}a_{i_j}\beta^{i_j} : \forall j, a_{i_j}\in F^\times\}.
$$
\end{theorem}

\begin{proof}
By Theorem~\ref{sigmadistinct} and Lemma~\ref{IJdistinct}, it suffices to partition $K(I)$ into equivalence classes under $\sim$ for each fixed $K/F$, $\sigma$ and $I\subset \mathcal{I}$. 

Let $\mathcal{C}_K$ be a fixed set of representatives for $F^\times/\normKFim$.  When $\zeta \notin \normKFim$, it follows that  $\mu_m\cap \normKFim = \{1\}$ and thus this $m$-element group must represent the quotient, allowing us to set $\mathcal{C}_K=\mu_m$.

Let $a\in K(I)$ and let $i_0<i_1<\cdots<i_k$ be the distinct elements of $I$.  Then
$a_{i_0} = cn$  for some unique $c\in \mathcal{C}_K$ and some $n\in \normKFim$.  Consequently, scaling $a$ by $n^{-1}$
gives $$
a \sim c\beta^{i_0} + \sum_{j=1}^{k}a_{i_j}\beta^{i_j}
$$
for some $a_{i_j}\in F^\times$ and $c\in \mathcal{C}_K$.
Furthermore, for any $s\in \Z$, we have
\begin{equation}\label{scalinggen}
\sigma^s\left(c\beta^{i_0}+ \sum_{j=1}^{k}a_{i_j}\beta^{i_j}\right)
= c\zeta^{si_0}\beta^{i_0} + \sum_{j=1}^{k}a_{i_j}\zeta^{si_j}\beta^{i_j}.
\end{equation}
If $\zeta\in \normKFim$, then we may scale this expression by $\zeta^{-si_0}\in \normKFim$ to produce an equivalent element of $K(I)$. It follows that the equivalence classes are parametrized by
$$
\{c\beta^{i_0} + \sum_{j=1}^{k}a_{i_j}\beta^{i_j} : (a_{i_1},\ldots, a_{i_k})\in F^{[\times k]}/\Delta_I\} = K(I;c),
$$
as $c$ varies over $\mathcal{C}_K$, as required.

If instead we have $\mathcal{C}_K=\mu_m$, then there are two cases.  If $i_0=0$, then the relations induced by \eqref{scalinggen} imply directly, as above, that every element $a\in K(I)$ is equivalent to a unique element of the set $K(I;c)$.  However, if $i_0>0$, then as $s$ varies over $\{1,2, \cdots, m\}$, the coefficient of $\beta^{i_0}$ in \eqref{scalinggen} varies over $\mathcal{C}_K$.  Consequently, every $a\in K(I)$ is equivalent to one for which the first nonzero coefficient $c$ is equal to $1$, and no two distinct elements of this form will be equivalent.  Taken together, this gives the set $\mathcal{S}(K)$ in this second case.
\end{proof}

\subsection{Application to local nonarchimedean fields}
From now on, let $F$ be again a local nonarchimedean field of residual characteristic $p$ and residue field of order $q$.  Let $m$ be  an odd prime distinct from $p$; then all extensions of degree $m$ of $F$ are tamely ramified.  Let us furthermore assume that $F$ contains a primitive $m$th root of unity, or equivalently, that $m|(q-1)$, where $q = \vert \kappa \vert$ is the order of the residue field of $F$.

Since $F^\times/(F^\times)^m \cong \Z/m\Z \times \Z/m\Z$, Kummer theory implies that we have precisely $m+1$ distinct Galois extensions of degree $m$, corresponding to all order-$m$ subgroups, and consequently $m+1$ collections of isomorphism classes, one for each such extension.

We now make the parametrization given in Theorem~\ref{Main:general} wholly explicit in this setting, by defining a set $\mathcal{C}_K$ of representatives of $F^\times/\normKFim$ for each degree $m$ extension of $F$, and determining the conditions under which  $\mu_m\subset \normKFim$.

\begin{proposition}\label{le:explicit}
Let $K$ be a cyclic Galois field extension of degree $m$ of a local nonarchimedean field $F$, such that $\mu_m\subset F$ and $m\neq p$.  Let $q$ denote the order of the residue field $\kappa$ of $F$.    Then $F^\times/\normKFim$ is represented by
$$
\mathcal{C}_K=\begin{cases}
\{1, \varpi, \varpi^2, \cdots, \varpi^{m-1}\} & \text{if $K/F$ is unramified;}\\
\mu_m(F) & \text{if $K/F$ is ramified and $m^2\nmid (q-1)$;}\\
\{1,\proot, \proot^2, \cdots, \proot^{m-1}\} & \text{if $K/F$ is ramified and $m^2|(q-1)$},
\end{cases}
$$
where $\proot\in \mu_{q-1}(F)$ is any choice of primitive $(q-1)$th root of unity in $F$.
\end{proposition}

\begin{proof}
Since $m\neq p$, Hensel's Lemma implies that the map $x\mapsto x^m$ is a bijection on $1+\PP$, and that $\mu_{q-1}(F)\cong \kappa^\times$ as in Example~\ref{eg:Teichmuller}. 

Thus via the map sending a pair $(\overline{a},\ell)\in \kappa^\times\times\Z$ to $a\varpi^\ell$, we have $\kappa^\times/(\kappa^\times)^m \times \Z/m\Z \cong F^\times/(F^\times)^m$.
Now $(F^\times)^m\subset \normKFim \subset F^\times$ and $\vert F^\times/\normKFim\vert = m$ by local class field theory.  When $K/F$ is unramified, the norm map surjects onto $\RR^\times$, so we may take our representatives $\mathcal{C}_K$ from powers of $\varpi$.  When $K/F$ is (totally) ramified, we instead have $\normKFim\cap \RR^\times  =(\RR^\times)^m$ so is suffices to identify a set of representatives from among the Teichm\"uller lifts of elements of $\kappa^\times/(\kappa^\times)^m$.

If $m^2\nmid(q-1)$, then no primitive $m$th root of unity can be an $m$th power, and thus we may take $\mathcal{C}_K=\mu_m$ as a preferred set of representatives.  Otherwise, one may take the first $m$ powers of any primitive element of $\mu_{q-1}(F) \cong \kappa^\times$ (or indeed, of any element whose order $k$ satisfies $mk \nmid(q-1)$).
\end{proof}

\subsection{An explicit example:  nonassociative cyclic algebras of degree $3$}\label{sec:degree3}

In this section, we specialize the results of the preceding section to degree $3$, to better illustrate the combinatorial and explicit nature of the parametrization.

\begin{theorem}\label{thm:mainII}
Let $F$ be a local nonarchimedean field of residual characteristic different from $3$ such that $F$ contains a primitive cube root of unity $\zeta$.  Write $\mu_3=\langle \zeta \rangle$ and set $\Delta_3 = \{ (a,a^{-1})\in F^\times\times F^\times | a\in \mu_3\}$.
Then the distinct isomorphism classes of nonassociative cyclic algebras of degree $3$ over $F$ are represented by $(K/F,\sigma,a)$  where $\sigma$ is a nontrivial element of $\Gal(K/F)$ and $K$ and $a\in \mathcal{S}(K)$ are determined as follows:
\begin{enumerate}[(a)]
\item $K=L_3$ is the unique unramified extension of $F$ and
\begin{align*}
\mathcal{S}(K) = \{r\beta, &r\beta^2, r\beta+s\beta^2, r+s\beta, r+s\beta^2, r+s_1\beta+s_2\beta^2 \; \vert \\
 &r\in \{1,\varpi,\varpi^2\}, s\in F^\times/\mu_3, (s_1,s_2)\in (F^\times\times F^\times)/\Delta_3\}.
\end{align*}
\item $K=F(\beta)$ is one of the three ramified extensions, where $\beta^3=u\varpi$ for some $u\in \RR^\times/(\RR^\times)^3$, and
\begin{enumerate}[(a)]
\item if $q\not\equiv 1 \mod 9$, then
$$
\mathcal{S}(K)= \{r+s_1\beta + s_2\beta^2 , \beta+s \beta^2, \beta^2 | r\in \mu_3, s\in F, (s_1,s_2)\in  (F^\times\times F^\times)/\Delta_3\};
$$
\item if $q\equiv 1\mod 9$, then let $F^\times/\normKFim = \RR^\times/(\RR^\times)^3=\{1,\gamma,\gamma^2\}$ and set
\begin{align*}
\mathcal{S}(K) = \{r\beta, &r\beta^2, r\beta+s\beta^2, r+s\beta, r+s\beta^2, r+s_1\beta+s_2\beta^2 \; \vert \\
 &r\in \{1,\gamma,\gamma^2\}, s\in F^\times/\mu_3, (s_1,s_2)\in (F^\times\times F^\times)/\Delta_3\}.
\end{align*}
\end{enumerate}
\end{enumerate}
\end{theorem}

\begin{proof}
Note that $\mu_3\subset \RR^\times$ implies that $3$ divides $q-1$ or $q\equiv 1\mod 3$, where $q$ is the order of the residue field of $F$.  Similarly, $\zeta$ is itself a cube if and only if $F^\times$ contains a primitive $9$th root of unity, which is the condition that $9|(q-1)$.
Here, $\mathcal{I} = \{ \{1\}, \{2\}, \{0,1\}, \{1,2\}, \{0,2\}, \{0,1,2\}\}$.  When $I=\{i\} \in \mathcal{I}$, we have $K(I,r)=\{r\beta^i\}$; when $I=\{i_0<i_1\}\in \mathcal{I}$, we have $K(I,r) = \{r\beta^{i_0} + s\beta^{i_1} : s\in F^\times/\mu_3\}$; and when $I=\{0,1,2\}$, we have $K(I,r) = \{r+s_1\beta+s_2\beta^2:(s_1,s_2)\in \Delta_{\{0,1,2\}}=\Delta_3\}.$
The statement now from Theorem~\ref{Main:general}.
\end{proof}

\section{Nonassociative cyclic algebras of degree four}\label{sec:degree4}

When the degree of the nonassociative cyclic algebra is not prime, its structure becomes more interesting.
Let $K/F$ be a cyclic Galois field extension of degree $m$ with Galois group $G={\rm Gal}(K/F)=\langle \sigma\rangle$ and $a\in K\smallsetminus F$.

 When $m$ is not prime, then there is a nonassociative subalgebra of $(K/F,\sigma,a)$ associated to each $1<s<m$ such that $(s,m)\neq 1$.  Namely, set $E = {\rm Fix}(\sigma^{s})$ to be the subfield of $K$ fixed by $\sigma^s$; then $(K/E,\sigma^s,a)$ is a cyclic algebra over $E$, that is nonassociative if $a\notin E$, and that is an $F$-subalgebra of $(K/F,\sigma,a)$.

In particular, if we let $H=\{\tau\in G\,|\,\tau(a)=a \}$ be the subgroup of the (cyclic) Galois group fixing $a$, and let $E= {\rm Fix}(H)$, then $H=\sigma^s$ for some divisor $s$ of $m$ and by \cite[Theorem 1]{P24}we have
 $${\rm Nuc}_r((K/F,\sigma, a))=(K/E,\sigma^s, a),$$
where this latter is now an associative cyclic algebra of degree $r$ over $E={\rm Fix}(\sigma^s)$ (and hence independent of the choice of generator $\sigma^s$).

 \color{black}

We briefly discuss the degree four case.

We first consider $F$ an arbitrary field of characteristic different from $2$.  Let $K/F$ be a cyclic Galois field extension of degree four with ${\rm Gal}(K/F)=\langle \sigma\rangle$.   Then $K$ has exactly one  quadratic subfield $E={\rm Fix}(\sigma^2)$.

Let $a\in K\smallsetminus F$ and consider $(K/F,\sigma,a)$.  This is a 16-dimensional algebra over $F$.  We know by Theorem~\ref{SteeleThm} that if $a\notin E$ then $(K/F,\sigma,a)$ is a division algebra.

Using Definition~\ref{D:cyclic} or the general results mentioned above, the nonassociative cyclic $E$-algebra $B=(K/E,\sigma^2, a)$ can be embedded as the  $F$-subalgebra generated by $1$ and $t^2$ in $(K/F,\sigma,a)$; it is $8$-dimensional over $F$.

\begin{lemma}
\begin{enumerate}[(a)]
\item If $a\not\in E$, then $B$ is a proper nonassociative quaternion division algebra over $E$, hence also a  nonassociative division algebra over $F$.

\item If  $a\in E$, then $B$  is the right nucleus of $A$.  In this case, it is an associative quaternion algebra over $E$, which is a division algebra if and only if $a\not\in \normKFim$.
Furthermore, then $B$ is a division algebra if and only if $(K/F,\sigma,a)$ is a division algebra.
\end{enumerate}
\end{lemma}

We give an explicit, but coarse, classification of these algebras in the case of local nonarchimedean fields as follows (note that these may not all be division algebras).

\begin{proposition}
Let $F$ be a local nonarchimedean field of residual characteristic different from $2$. Then every nonassociative cyclic algebra of degree four over $F$
is of one of the following types, and moreover, algebras corresponding to distinct field extensions and distinct generators of the Galois group are nonisomorphic:
\begin{enumerate}
\item $(L_4/F,\sigma,a)$, corresponding to the unique unramified extension of degree $4$.  Here $B=(L_4/F(\sqrt{\varepsilon}),\sigma^2, a)$.

\item  $(K_0/F,\sigma,a)$, corresponding to the extension $K_0 = E(\sqrt{\varepsilon_E\varpi})$ where $E=F(\sqrt{\varepsilon})$ and $\varepsilon_E$ is a nonsquare of $\RR_E^\times\smallsetminus \RR^\times$.   Here $B=(K_0/E,\sigma^2, a)$.
 \item  If $-1\in (F^\times)^2$: for $i\in \{1,2,3,4\}$, $(K_i/F,\sigma,a)$, corresponding to the extension $K_i=F(\sqrt[4]{\varepsilon^i\varpi})$.  The intermediate subfield is $E=F(\sqrt{\varpi})$ if $i$ is even and  $E=F(\sqrt{\varepsilon\varpi})$ if $i$ is odd. Here $B=(K_i/E,\sigma^2, a)$.
 \end{enumerate}
\end{proposition}

\begin{proof}
Let $A$ be a nonassociative cyclic algebra  of degree four over $F$. Since the left nucleus of $A$ is invariant under isomorphism, it suffices to generate a list of distinct cyclic Galois extensions $K$ of $F$ of degree $4$. Moreover, by Theorem~\ref{sigmadistinct}, we know that for the two distinct choices of generator $\sigma$ of $\Gal(K/F)$ the resulting algebras will be nonisomorphic.

It can be shown that $F$ admits $6$ cyclic Galois extensions of degree four if $-1\in (F^\times)^2$ and only two otherwise.
The first of these is the unique unramified extension $L_4$, with intermediate subfield $E=L_2$.   The second is a partially ramified extension $K_0$, obtained as the quadratic extension of $E=L_2$ by an element of the form $u\varpi$, where $u$ is chosen as a nonsquare of $\RR_{L_2}^\times$ that does not lie in $\RR^\times$.

When $-1 \in (F^\times)^2$, then $F$ contains a primitive $4$th root of unity, and there are correspondingly four additional Galois totally ramified extensions, given as $K_i=F(\sqrt[4]{\varepsilon^i\varpi})$.  Their unique intermediate subfield is $E=F(\sqrt{\varpi})$ if $i$ is even and $E=F(\sqrt{\varepsilon\varpi})$ if $i$ is odd.
\end{proof}

Thus we count four different types for each such field $F$ if $-1\not\in (F^\times)^2$ and twelve if $-1\in (F^\times)^2$.

While in the associative case, all cyclic algebras of degree $m$ over $F$ are isomorphic to an algebra of the type $(L_m/F,\sigma,a)$ for some $a\in F^\times$, this proposition shows once again that the situation in the nonassociative setting is more varied.
\\\\
 \emph{Acknowledgements:} This paper was written while the second author was a visitor at the University of Ottawa in the spring of 2024. She would like to thank the Department of Mathematics and Statistics for its hospitality and congenial atmosphere.

\providecommand{\bysame}{\leavevmode\hbox to3em{\hrulefill}\thinspace}
\providecommand{\MR}{\relax\ifhmode\unskip\space\fi MR }
\providecommand{\MRhref}[2]{%
  \href{http://www.ams.org/mathscinet-getitem?mr=#1}{#2}
}
\providecommand{\href}[2]{#2}

\end{document}